\documentclass[twoside]{article}

\usepackage[accepted]{aistats2025arxiv}
\usepackage[utf8]{inputenc} 
\usepackage[T1]{fontenc}    
\usepackage{hyperref}       
\usepackage{amsfonts}       
\usepackage{nicefrac}       
\usepackage{microtype}      
\usepackage{amsthm, amsmath, amssymb, mathtools} 
\usepackage{cleveref} 
\usepackage{graphicx}
\graphicspath{{plots/}}

\usepackage[dvipsnames]{xcolor}
\usepackage{pifont}

\usepackage{amsthm}

\usepackage{algorithm}
\usepackage{algcompatible}

\usepackage[colorinlistoftodos,bordercolor=orange,backgroundcolor=orange!20,linecolor=orange,textsize=scriptsize]{todonotes}

\newtheorem{assumption}{Assumption}
\newtheorem{lemma}{Lemma}

\newtheorem{theorem}{Theorem}

\newtheorem{corollary}{Corollary}

\theoremstyle{plain}

\newtheorem{remark}[theorem]{Remark}

\theoremstyle{definition}

\newtheorem{definition}[theorem]{Definition}
\DeclareMathOperator*{\argmin}{argmin}

\DeclareMathOperator*{\Argmin}{Argmin}

\newcommand{\lin}[1]{\left\langle #1\right\rangle} 

\newcommand{\ExpBr}[1]{\mathbb{E}\left[#1\right]}
\usepackage{xspace}
\newcommand{\algname}[1]{\textsf{\textcolor{ForestGreen}{#1}}}

\newcommand{\prox}{\mathop{\mathrm{prox}}\nolimits}

\newcommand{\avein}{\frac{1}{n}\sum_{i=1}^n}

\newcommand{\RR}{\mathbb{R}}

\newcommand{\cO}{{\cal O}}

\newcommand{\bI}{\mathbf{I}}
\newcommand{\bA}{\mathbf{A}}

\newcommand{\bQ}{\mathbf{Q}}

\newcommand{\del}[1]{}

\newcommand{\eqdef}{\stackrel{\text{def}}{=}}

\usepackage[round]{natbib}

\begin{document}

%
\runningtitle{Speeding Up Stochastic Proximal Optimization In The High Hessian Dissimilarity Setting}

%

\twocolumn[

\aistatstitle{Speeding up Stochastic Proximal Optimization \\in the High Hessian Dissimilarity Setting}

\aistatsauthor{ Elnur Gasanov \And Peter Richt\'arik}

\aistatsaddress{ KAUST\And  KAUST} 
]

\begin{abstract}
  Stochastic proximal point methods have recently garnered renewed attention within the optimization community, primarily due to their desirable theoretical properties. Notably, these methods exhibit a convergence rate that is independent of the Lipschitz smoothness constants of the loss function, a feature often missing in the loss functions of modern ML applications. In this paper, we revisit the analysis of the Loopless Stochastic Variance Reduced Proximal Point Method (\algname{L-SVRP}). Building on existing work, we establish a theoretical improvement in the convergence rate in scenarios characterized by high Hessian dissimilarity among the functions. Our concise analysis, which does not require smoothness assumptions, demonstrates a significant improvement in communication complexity compared to standard stochastic gradient descent.
\end{abstract}

\section{INTRODUCTION}

Machine Learning (ML) has become integral to the advancement of modern applications across diverse sectors, including manufacturing~\citep{manufacturing, predictive_maintenance}, healthcare~\citep{healthcare_dl, maleki2024role}, finance~\citep{goodell2021artificial}, retail~\citep{oosthuizen2021artificial}, agriculture~\citep{agriculture}, and beyond. This widespread integration arises from machine learning's capacity to analyze vast amounts of data, identify patterns, and leverage them for predictive insights.

Many machine learning problems admit a rigorous mathematical formulation. A prominent example is Expected Risk Minimization (ERM)~\citep{shalev2014understanding}, which is mathematically represented as the minimization problem:
\begin{equation}\label{eq:finite_sum}
\min\limits_{x \in\RR^d} \left\{f(x) \eqdef \avein f_i(x)\right\},
\end{equation}
where $n$ denotes the number of data points, $x \in \RR^d$ represents the model parameters, and $f_i(x)$ is the loss function associated with $i$-th data point. It is worth noting that the same formulation~\eqref{eq:finite_sum} can be interpreted in the context of Federated Learning~\citep{FEDLEARN, fl_tianli}, where $n$ refers to the number of clients, and $f_i(x)$ represents the loss function associated with data of client $i$.

Despite its simple form, problem~\eqref{eq:finite_sum} is an unusual, rich source of research problems, exhibiting a wide variety of properties and providing convergence bounds that depend on the assumptions imposed on $f_i$~\citep{tight_bounds_smoothness}. To focus our analysis amid this variability, we rely on two assumptions commonly found in the literature: Hessian similarity and strong convexity.

\begin{assumption}[Hessian similarity]~\label{as:hess_sim_finite_sum}
There exists a constant $\delta \geq 0$ such that for all vectors $x, y \in \mathbb{R}^d$, the following inequality holds:
\begin{equation}\label{eq:hess_sim_finite_sum}
\avein \|\nabla f_i(x) - \nabla f_i(y) - [\nabla f(x) - \nabla f(y)]\|^2 \leq \delta^2\|x - y\|^2.
\end{equation}
\end{assumption}

\begin{assumption}[Strong convexity]~\label{as:strong_cvx}
Each function $f_i$ is $\mu$-strongly convex, meaning that for all vectors $x, y \in \mathbb{R}^d$, the following inequality holds for some $\mu \geq 0$:
	\begin{equation}
		f_i(y) \geq f_i(x) + \lin{\nabla f_i(x), y - x} + \frac{\mu}{2}\|x-y\|^2.
	\end{equation}
\end{assumption}

It is worth noting that when $\mu = 0$, this assumption simplifies to the standard definition of convexity.

\Cref{as:hess_sim_finite_sum} captures the similarity condition first introduced by \citet{dane} (Lemma 1) to analyze the convergence rate of their \algname{DANE} algorithm for quadratic functions. This assumption was later independently rediscovered in \citep{szlendak2022permutation}. Furthermore, it serves as a second-order heterogeneity assumption~\citep{woodworth_soh, saber}, which, in some cases, depends on the input data but not on the data labels. Recently, \Cref{as:hess_sim_finite_sum} has gained attention as a potentially advantageous alternative to the often restrictive Lipschitz smoothness assumption in Deep Learning~\citep{non_smooth_dl, goodfellow2016deep}. The Lipschitz smoothness assumption posits that for all vectors $x, y \in \mathbb{R}^d$, the norm of the difference between gradients is bounded by a constant $L > 0$ multiplied by the distance between $x$ and~$y$:
$$
\|\nabla f(x) - \nabla f(y)\| \leq L \|x - y\|.
$$
For instance, \algname{SCAFFOLD}~\citep{scaffold} leveraged Hessian similarity for quadratic functions, eliminating the dependence on the smoothness constant in communication complexity. Subsequently, the \algname{SONATA} algorithm~\citep{sonata} extended this result to strongly convex functions.
Recent studies have further built upon~\Cref{as:hess_sim_finite_sum} in the context of ERM, leading to the development of  more efficient and even accelerated algorithms~\citep{kovalev2022optimal, khaled2023faster, svrs, stich_dane, stabilized_dane, saber}.

We emphasize that our analysis does not rely on the assumption of Lipschitz smoothness. Moreover, smoothness directly implies Assumption~\ref{as:hess_sim_finite_sum}, as demonstrated in the following derivation:
\begin{align*}
& \avein \|\nabla f_i(x) - \nabla f_i(y) - [\nabla f(x) - \nabla f(y)]\|^2  \\
& \overset{\eqref{eq:bias_variance_decomposition}}{=} \quad \avein \|\nabla f_i(x) - \nabla f_i(y)\|^2 - \|\nabla f(x) - \nabla f(y)\|^2 \\
& \leq \quad  \avein \|\nabla f_i(x) - \nabla f_i(y)\|^2 \\
& \leq \quad \avein L_i^2 \| x - y\|^2,
\end{align*}
where the first equality follows from the variance decomposition formula. Here, $L_i$ denotes the smoothness constant of function $f_i$. Our derivation above implies that $\delta^2 \leq \avein L_i^2$. Notably, in the  inequality on the third line, we discard the potentially large term $\left\| \nabla f(x) - \nabla f(y) \right\|^2$, indicating that $\delta$ can, in practice, be much smaller than the smoothness constant. This observation highlights the extra power of our assumption in comparison to the traditional $L$-smoothness condition.

The strong convexity assumption is a cornerstone in optimization theory and has been widely utilized across numerous algorithms~\citep{nesterov2018lectures}. In particular, this assumption ensures that the function $f(x)$ is $\mu$-strongly convex, guaranteeing the existence of a unique solution $x_\ast \in \mathbb{R}^d$ to the ERM problem when $\mu > 0$. In the case where $\mu = 0$, $x_\ast$ refers to any point in $\Argmin f$.

\subsection{Contributions}
This paper is dedicated to a more in-depth analysis of the Loopless Stochastic Variance Reduced Proximal point method (\algname{L-SVRP}), which was initially introduced in the work of \citet{khaled2023faster}, inspired by the~\algname{L-SVRG} of~\cite{l_svrg}. Our primary contributions are outlined as follows:
\begin{enumerate}  
	\item Under Assumptions~\ref{as:hess_sim_finite_sum} and~\ref{as:strong_cvx}, we derive a novel iteration complexity bound of $O\left(\left(n\frac{\delta}{\mu} + n\right)\log{\frac{\|x_0 - x_\ast\|^2}{\varepsilon}}\right)$ for the \algname{L-SVRP} algorithm, outlined in~\Cref{alg:l_svrp}. Importantly, our result demonstrates a significant improvement over the iteration and communication complexity bounds previously established by \citet{khaled2023faster}. In particular, our complexity bound is asymptotically sharper in the $\frac{\delta}{\mu} \geq n$ regime, surpassing the earlier results obtained by~\citet{khaled2023faster, unified_theory_sppa}. Moreover, unlike previous analyses, such as \citet{vrt_sppa}, our approach does not rely on the smoothness assumption. Additionally, we observe that \algname{L-SVRP} can be applied to monotone inclusion problems, as demonstrated by \citet{sadiev2024stochasticproximalpointmethods}.
	\item We provide experimental validation to support our theoretical findings. Specifically, our experiments confirm the linear convergence rate of the \algname{L-SVRP} algorithm and verify the accuracy of the derived upper bounds on its convergence rate.
\end{enumerate}

The remainder of the paper is structured as follows. In~\Cref{sec:importance_loopless_prox}, we guide the reader through a detailed exploration of the significance of loopless methods and proximal algorithms, tracing their development and relevance in optimization. Following this, we introduce the \algname{L-SVRP} algorithm, which belongs to both the loopless and proximal categories. In~\Cref{sec:theory}, we delve into our novel theoretical contributions, offering a thorough comparison between our results and existing iteration complexity rates. Finally, in~\Cref{sec:exps}, we provide a thorough presentation of our experimental results.

\section{\uppercase{The Importance of Loopless and Proximal Methods in Optimization}}\label{sec:importance_loopless_prox}
We begin by exploring the dual nature—both the advantages and limitations—of early stochastic variance reduction methods, which have played a pivotal role in the development of modern optimization algorithms.

\subsection{Stochastic Gradient Descent}

A natural approach to solving the ERM problem~\eqref{eq:finite_sum} is to employ the Stochastic Gradient Descent (\algname{SGD})~\citep{robbins1951stochastic} algorithm. At each iteration $k$ of \algname{SGD}, the algorithm randomly selects an index $i_k \in \{1, \dots, n\} $ and updates the parameters according to the rule
\begin{equation}\label{eq:sgd}
x_{k+1} = x_k - \gamma \nabla f_{i_k}(x_k),
\end{equation}
where $\gamma$ is a positive stepsize. A key issue with \algname{SGD} is its convergence behavior: instead of converging to the exact solution of the finite sum problem~\eqref{eq:finite_sum}, it converges to a neighborhood of the optimal solution only. The size of this neighborhood is proportional to the variance of the stochastic gradient and the stepsize $\gamma$, while being inversely proportional to the strong convexity parameter $\mu$, as shown by \cite{gower2019sgd}. Consequently, the stepsize $\gamma$ must be reduced to match the target precision $\varepsilon$, leading to very slow convergence of $\widetilde{\cO}\left(\frac{\max L_i}{\mu} + \frac{\sigma_{\ast}^2}{\mu^2 \varepsilon}\right)$, where $\sigma_\ast^2 = \avein \|\nabla f_i(x_\ast)\|^2$. To overcome these limitations, variance reduction techniques have been developed, leading to faster convergence rates. A prominent example of such methods is the Stochastic Variance Reduced Gradient~\citep{svrg} algorithm.

\subsection{Stochastic Variance Reduced Gradient}

Stochastic variance reduced gradient methods, such as \algname{SVRG}~\citep{svrg}, address the limitations of \algname{SGD} by reducing the variance of stochastic gradient estimates, enabling a significant improvement in convergence speed.

\begin{algorithm}
	\begin{algorithmic}
		\STATE \textbf{Input:} update frequency $m \in \mathbb{N}$, stepsize $\gamma > 0$, initial iterate \(x_0 \in \RR^d\), reference point $w_0 \in \RR^d$, number of iterations $K \in \mathbb{N}$
		\FOR{$k = 0, 1, 2, \dots, K - 1$}
		\STATE Uniformly at random choose $i_k \in \{1, 2, \dots, n\}$
		\STATE $x_{k+1} = x_{k} - \gamma (\nabla f_{i_k}(x_{k}) - \nabla f_{i_k}(w_k) + \nabla f(w_k))$
		\STATE $w_{k+1} = 
		\begin{cases}
			x_{k+1} & \text{if } (k+1) \ \% \ m = 0, \\
			w_k & \text{otherwise.}
		\end{cases}$
		\ENDFOR
		\STATE \textbf{Output:} $x_K$
	\end{algorithmic}
	\caption{\algname{SVRG}}
	\label{alg:svrg}
\end{algorithm}

Unlike~\algname{SGD}, the Stochastic Variance Reduced Gradient (\algname{SVRG}) method~\citep{svrg}, as detailed in Algorithm~\ref{alg:svrg}, progressively estimates the gradient shift $ \nabla f_i(x_\ast) $. This adjustment allows \algname{SVRG} to converge precisely to the solution of the original problem~\eqref{eq:finite_sum} with an iteration complexity of $\widetilde{\cO}\left( n + \frac{L}{\mu} \right)$, assuming an optimal inner loop size  $m$. This improvement alone represents a significant acceleration compared to the iteration complexity of full batch gradient descent, which is $ \widetilde{\cO}\left(n\frac{L}{\mu}\right)$.

However, practical implementation introduces challenges. The optimal inner loop size $m$, as suggested by theoretical analysis, depends on the strong convexity parameter $\mu$, which is often unknown or loosely estimated in real-world scenarios.

This issue was addressed by the introduction of \algname{L-SVRG}, the loopless variant of \algname{SVRG} \citep{l_svrg, l_svrg_original, s2gd}. The key trick is to eliminate the inner loop entirely through randomization.

\subsection{Escaping the Loop}
Loopless SVRG (\algname{L-SVRG}), as presented in~\Cref{alg:l_svrg}, closely resembles the original \algname{SVRG} algorithm, but eliminates the inner loop by employing a randomized update for the reference point $w_k$.

\begin{algorithm}
	\begin{algorithmic}
		\STATE \textbf{Input:} probability $p \in (0, 1]$, stepsize $\gamma > 0$, initial iterate $x_0 \in \RR^d$, reference point $w_0 \in\RR^d$, number of iterations $K \in \mathbb{N}$
		\FOR{$k = 0, 1, 2, \dots, K - 1$}
		\STATE Uniformly at random choose $i_k \in \{1, 2, \dots, n\}$
		\STATE $x_{k+1} = x_{k} - \gamma (\nabla f_{i_k}(x_{k}) - \nabla f_{i_k}(w_k) + \nabla f(w_k))$
		\STATE $w_{k+1} = \begin{cases}
		x_{k+1} & \text{with probability } p\\
		w_k &  \text{with probability } 1 - p
		\end{cases}$
		\ENDFOR
		\STATE \textbf{Output:} $x_K$
	\end{algorithmic}
	\caption{\algname{L-SVRG}}
	\label{alg:l_svrg}
\end{algorithm}

The iteration complexity of this algorithm, when $p = \frac{1}{n}$, matches that of \algname{SVRG} with the optimal inner loop size $m$, resulting in the complexity $ \widetilde{\cO}\left(n + \frac{L}{\mu}\right)$. However, in \algname{L-SVRG}, the expected inner loop length, $\frac{1}{p}$, depends only on the number of samples or clients $n$, which enhances its practical applicability. Moreover, the convergence analysis becomes more straightforward and easier to interpret.

The application of the loopless framework extends to several other algorithmic pairs. For example, \algname{SARAH}~\citep{sarah} is enhanced by its loopless variant, which incorporates probabilistic mechanism in \algname{PAGE}~\citep{page} and quantization in \algname{MARINA}~\citep{marina}. Notably, \algname{PAGE} achieves optimal complexity in smooth non-convex settings. Similarly, the stochastic algorithm \algname{ESGDA}~\citep{esgda}, designed for saddle point optimization problems, simplifies to \algname{RSGDA}~\citep{rsgda}, making the analysis more tractable. Finally, some algorithms, such as \algname{ADIANA}~\citep{adiana}, \algname{CANITA}~\citep{canita}, and \algname{ANITA}~\citep{anita}, were developed solely in their loopless form.

To summarize, the loopless approach offers several advantages for algorithm design,

\begin{enumerate}
\item it simplifies the structure of otherwise more complex algorithms,
\item it facilitates a more simplified convergence analysis,
\item it often results in tighter convergence bounds, these bounds are never worse than existing approaches and are sometimes significantly sharper,
\item it is versatile and can be applied across various optimization settings.
\end{enumerate}

\subsection{Proximal Algorithms}
Proximal algorithms~\citep{parikh2013proximal, condat2023proximalsplittingalgorithmsconvex} are a fundamental class of optimization methods, particularly useful for handling non-smooth and composite objective functions. They leverage the concept of the proximity operator, defined via:
\begin{equation}\label{eq:prox_op_def}
\prox_{g}(x) \eqdef \argmin\limits_{x'\in\RR^d}\left\{g(x')+\frac{1}{2}\|x' - x\|^2\right\},
\end{equation}
where $g: \RR^d \rightarrow \RR \cup \{+\infty\}$ is an extended real-valued function. If $g(x)$ is a proper, closed, convex function, then the minimizer of $g(x')~+~\frac{1}{2}\|x' - x\|^2$ exists and is unique~\citep{bauschke2011convex}. Moreover, if $g$ is differentiable, the following equivalence holds:
\begin{equation}\label{eq:prox_op_equiv}
y = \prox_{\gamma g}(x) \ \Leftrightarrow \ y + \gamma \nabla g(y) = x.
\end{equation}
Further properties and formal statements regarding the proximity operator are provided in the appendix.

The simplest stochastic proximal algorithm is the Stochastic Proximal Point Method (\algname{SPPM}) by~\citet{duchi}. At each iteration $k$, for a randomly selected $i_k \in \{1, \dots, n\}$, the following update is performed:
\begin{equation}
x_{k+1} = \prox_{\gamma f_{i_k}}(x_k),
\end{equation}
or equivalently, using \eqref{eq:prox_op_equiv}:

\begin{equation}
	x_{k+1} = x_k - \gamma \nabla f_{i_k}(x_{k+1}).
\end{equation}

When compared to~\eqref{eq:sgd}, we observe that \algname{SPPM} is virtually identical to~\algname{SGD}, with one seemingly minor yet conceptual difference: the gradient is computed at the new point $x_{k+1}$. Proximal algorithms, including \algname{SPPM}, are implicit methods, meaning the update involves solving an equation where the new iterate appears on both sides. This implicitness leads to enhanced stability in practice~\citep{ryu2016stochastic}, drawing parallels with the stability observed in implicit numerical solutions of ordinary differential equations~\citep{runge_kutta}. Moreover, the convergence of proximal algorithms does not depend on the smoothness properties of the minimized functions~\citep{prox_unified}, which makes them appealing for optimizing deep learning loss objectives where the smoothness assumption may not hold~\citep{non_smooth_dl}.

In recent years, there has been growing interest in applying proximal algorithms within Federated Learning (FL). \algname{FedAvg}~\citep{FEDLEARN, 44822}, a well-known FL algorithm used in real-world applications, operates as follows: a client, upon receiving an iterate $x_k$ from the server, performs \textit{several} steps of \algname{SGD} on its local data. The final iterate is then transmitted back to the server for aggregation, and this process repeats iteratively. Performing multiple local iterations can be viewed as approximating the proximal operator applied to the local objective function, akin to solving problem~\eqref{eq:prox_op_def} for some surrogate local function. Some researchers propose replacing local steps with the proximity operator to deepen our understanding of local methods and obtain faster convergence rates~\citep{fedprox, pfl_moreau, fedexp, fedexprox}. Other researchers interpret the aggregation step as a proximity operator, achieving state-of-the-art communication complexity for this class of methods~\citep{proxskip, 5gcs, proxskip_tighter, vr_proxskip}.

Based on our discussion, we can summarize the key points as follows:

\begin{enumerate}
	\item Proximal algorithms are particularly relevant in Federated Learning, owing to their connection with local computations and distributed aggregation schemes.
	\item They exhibit enhanced stability in practice, particularly in stochastic settings.
	\item Their convergence does not rely on smoothness assumptions, making them suitable for nonsmooth optimization problems.
\end{enumerate}

\subsection{Loopless and Proximal}

Motivated by the reasons outlined in the previous two sections, we focus our research on an algorithm that lies at the intersection of both loopless and proximal methods: the Loopless Stochastic Variance Reduced Proximal Point method (\algname{L-SVRP}).

\begin{algorithm}
	\begin{algorithmic}
		\STATE \textbf{Input:} probability $p \in (0, 1]$, stepsize $\gamma > 0$, initial iterate $x_0 \in \RR^d$, reference point $w_0 \in\RR^d$,  number of iterations $K \in \mathbb{N}$
		\FOR{$k = 0, 1, 2, \dots, K - 1$}
		\STATE Uniformly at random choose $i_k \in \{1, 2, \dots, n\}$
		\STATE $x_{k+1} = \prox_{\gamma f_{i_k}}(x_k + \gamma (\nabla f_{i_k}(w_k) - \nabla f(w_k))$
		\STATE $w_{k+1} = \begin{cases}
			x_{k+1} & \text{with probability } p\\
			w_k &  \text{with probability } 1 - p
		\end{cases}$
		\ENDFOR
		\STATE \textbf{Output:} $x_K$
	\end{algorithmic}
	\caption{\algname{L-SVRP}}
	\label{alg:l_svrp}
\end{algorithm}

\algname{L-SVRP} can be viewed as the implicit counterpart of the previously discussed \algname{L-SVRG} algorithm. Specifically, the main update steps of both algorithms are related as follows:
\begin{align*}
	x_{k+1} &\overset{\text{Alg.~\ref{alg:l_svrg}}}{=} x_{k} - \gamma \left( \nabla f_{i_k}(x_{k}) - \nabla f_{i_k}(w_k) + \nabla f(w_k) \right), \\
	x_{k+1} &\overset{\text{Alg.~\ref{alg:l_svrp}}}{=} x_{k} - \gamma \left( \nabla f_{i_k}(x_{k+1}) - \nabla f_{i_k}(w_k) + \nabla f(w_k) \right),
\end{align*}
where, in deriving the main update rule of Algorithm~\ref{alg:l_svrp}, we utilize the equivalence from~\eqref{eq:prox_op_equiv}. This relationship highlights that \algname{L-SVRP} incorporates the gradient at the future iterate $x_{k+1}$ in its update rule, making it an implicit method. In contrast, \algname{L-SVRG} uses the gradient at the current iterate $x_k$, characterizing it as an explicit method.

Although several alternative algorithms, such as \algname{SVRS}~\citep{svrs} and \algname{SABER}~\citep{saber}, have recently been proposed—employing similar combinations of loopless and proximal ideas and achieving superior convergence rates under Assumption~\ref{as:hess_sim_finite_sum}—we wish to emphasize that the goal of this work is to improve the convergence rate of the existing \algname{L-SVRP} algorithm. To the best of our knowledge, we are the first to revisit the theoretical analysis of \algname{L-SVRP}. Other algorithms, such as~\algname{RandProx}~\citep{randprox}, while sharing similar properties, are analyzed under the assumption of smoothness.

\section{THEORETICAL ANALYSIS}\label{sec:theory}

To establish contraction, we define a Lyapunov function that combines two components: $\mathbb{E}\left[ \| x_k - x_\ast \|^2 \right]$ and $\mathbb{E}\left[ \| w_k - x_k \|^2 \right]$. Specifically, we express it as:

\[
\Lambda_k \coloneqq \|x_k - x_\ast\|^2 + c\|w_k - x_k\|^2,
\]

where \(c > 0\) is a positive scalar determined later, \(x_k\) and \(w_k\) are the iterates of Algorithm~\ref{alg:l_svrp}, and $x_\ast \in \Argmin f$ represents any solution to problem~\eqref{eq:finite_sum}.

In the course of our derivations, we will make use of the following conditional expectation:
$$
\bar{x}_{k+1} = \ExpBr{x_{k+1} \mid x_k, w_k}.
$$

\subsection{Bounding the First Term of the Lyapunov Function}
To begin the analysis, we establish a bound $\mathbb{E}\left[ \| x_k - x_\ast \|^2 \right]$, as presented in Lemma~\ref{lm:lsvrp_lyap_first}.

\begin{lemma}\label{lm:lsvrp_lyap_first}
	Let~\Cref{as:strong_cvx} (strong convexity) hold. Let us define 
	$
	\psi_i(x) \eqdef f_i(x)-f(x).
	$
	Then, for one iteration of Algorithm~\ref{alg:l_svrp}, the following inequality holds:
	\begin{equation}\label{eq:lsvrp_lyap_first}
		\begin{aligned}
			& \ExpBr{\|x_{k+1} - x_\ast\|^2 | x_k, w_k} \leq  \frac{1}{1 + \mu\gamma} \|x_k - x_\ast\|^2 \\
			& - \ExpBr{\|x_{k+1} - \bar{x}_{k+1}\|^2} - \frac{1}{1 + \mu\gamma}\|\bar{x}_{k+1} - x_k\|^2\\
			& - \frac{2}{\mu + \frac{1}{\gamma}}  \ExpBr{\lin{\nabla \psi_{i_k}(\bar{x}_{k+1}) - \nabla \psi_{i_k}(w_k), x_{k+1} - \bar{x}_{k+1}} | x_k, w_k}.
		\end{aligned}
	\end{equation}
\end{lemma}

We defer the proof of this lemma to the appendix. 

The most challenging term on the right-hand side of inequality~\eqref{eq:lsvrp_lyap_first} is the last inner product. Utilizing Assumption~\ref{as:hess_sim_finite_sum}, we can bound this term by a quantity proportional to $\| w_k - \bar{x}_{k+1} \|^2$. To eliminate this term in our combined Lyapunov analysis, we introduce Lemma~\ref{lm:lsvrp_lyap_second}, which provides a compensating term that cancels this quantity.

\subsection{Bounding the Second Term of the Lyapunov Function}

To complement the analysis, we derive a bound for the second term in the Lyapunov function, $\mathbb{E}\left[\| w_{k+1} - x_{k+1} \|^2\right]$.

\begin{lemma}\label{lm:lsvrp_lyap_second}
	For the iterates of Algorithm~\ref{alg:l_svrp}, and for any $\xi \in [0, 1]$ and $\zeta > 0$, the following inequality holds:
	\begin{equation}\label{eq:lsvrp_lyap_second}
		\begin{aligned}
			&\ExpBr{\|w_{k+1} - x_{k+1}\|^2 \mid x_k, w_k} \\
			&\leq (1 - p(1 - \xi))(1 + \zeta) \|w_k - x_k\|^2 \\
			& \quad  + (1 - p(1 - \xi))(1 + \zeta^{-1}) \|x_k - \bar{x}_{k+1}\|^2 \\
			&  \quad - p\xi \|w_k - \bar{x}_{k+1}\|^2 \\
			& \quad  + (1 - p) \ExpBr{\|\bar{x}_{k+1} - x_{k+1}\|^2 \mid x_k, w_k}.
		\end{aligned}
	\end{equation}
\end{lemma}

\begin{proof}
	The proof follows directly. Consider that
	\begin{equation}\label{eq:lsvrp_aux2}
		\begin{aligned}
		&\ExpBr{\|w_{k+1} - x_{k+1}\|^2 | x_{k+1}, x_k, w_k} \\
		&= p \|x_{k+1} - x_{k+1}\|^2 + (1 - p)\|w_k - x_{k+1}\|^2 \\
		&= (1 - p)\|w_k - x_{k+1}\|^2.
		\end{aligned}
	\end{equation}
	The remainder involves a series of straightforward bounds and expansions, as detailed below:
	\begin{equation}\label{eq:lsvrp_second_lyapunov_term_aux_1}
	\begin{aligned}
		&\ExpBr{\|w_{k+1} - x_{k+1}\|^2 | x_k, w_k} \\ 
		& = \ExpBr{\ExpBr{\|w_{k+1} - x_{k+1}\|^2 | x_{k+1}, x_k, w_k} | x_k, w_k} \\
		&=(1 - p)\ExpBr{\|w_k - x_{k+1}\|^2 | x_k, w_k} \\
		& = (1 - p)\ExpBr{\|w_k - \bar{x}_{k+1} +  \bar{x}_{k+1} - x_{k+1}\|^2 | x_k, w_k} \\
		& =  (1 - p) \biggl[ \|w_k - \bar{x}_{k+1}\|^2 +  \ExpBr{\|\bar{x}_{k+1} - x_{k+1}\|^2 | x_k, w_k} \\
		&  \quad + 2\ExpBr{\lin{w_k - \bar{x}_{k+1}, \bar{x}_{k+1} - x_{k+1} }| x_k, w_k}\biggr],
	\end{aligned}
	\end{equation}
	where the first equality follows from the tower property, and the second equality is obtained using~\eqref{eq:lsvrp_aux2}.
	
	Noting that, due to the linearity of expectation and the definition $\bar{x}_{k+1} = \ExpBr{x_{k+1}| x_k, w_k}$, we have
	\begin{equation}\label{eq:lsvrp_second_lyapunov_term_aux_2}
	\begin{aligned}
		&\ExpBr{\lin{w_k - \bar{x}_{k+1}, \bar{x}_{k+1} - x_{k+1} }| x_k, w_k} \\
		&= \lin{w_k - \bar{x}_{k+1}, \ExpBr{\bar{x}_{k+1} - x_{k+1} | x_k, w_k}} \\
		&= \lin{w_k - \bar{x}_{k+1}, 0} = 0.
	\end{aligned}
	\end{equation}
	Substituting this equality into~\eqref{eq:lsvrp_second_lyapunov_term_aux_1}, we obtain
	\begin{eqnarray*}
		&&\ExpBr{\|w_{k+1} - x_{k+1}\|^2 | x_k, w_k}  \\
		& \overset{\eqref{eq:lsvrp_second_lyapunov_term_aux_1}+\eqref{eq:lsvrp_second_lyapunov_term_aux_2}}{=} &  (1 - p) [ \|w_k - \bar{x}_{k+1}\|^2 \\
		&& +  \ExpBr{\|\bar{x}_{k+1} - x_{k+1}\|^2 | x_k, w_k} ]\\
		& = & (1 - p(1 - \xi))  \|w_k - \bar{x}_{k+1}\|^2 \\
		& & - p\xi \|w_k - \bar{x}_{k+1}\|^2 \\
		& & + (1 - p) \ExpBr{\|\bar{x}_{k+1} - x_{k+1}\|^2 | x_k, w_k} \\
		& = & (1 - p(1 - \xi))  \|w_k - x_k + x_k - \bar{x}_{k+1}\|^2 \\
		& & - p\xi \|w_k - \bar{x}_{k+1}\|^2 \\
		& & + (1 - p) \ExpBr{\|\bar{x}_{k+1} - x_{k+1}\|^2 | x_k, w_k} \\
		& \leq & (1 - p(1 - \xi))(1 + \zeta) \|w_k - x_k\|^2 \\
		& & + (1 - p(1 - \xi))(1 + \zeta^{-1}) \|x_k - \bar{x}_{k+1}\|^2 \\
		& &  - p\xi \|w_k - \bar{x}_{k+1}\|^2 \\
		& & + (1 - p) \ExpBr{\|\bar{x}_{k+1} - x_{k+1}\|^2 | x_k, w_k}.
	\end{eqnarray*}
	In the final inequality, we utilize the fact that for any vectors $a, b \in \RR^d$ and any constant $\zeta > 0$, the following bound holds: $\|a + b\|^2 \leq (1 + \zeta)\|a\|^2 + (1 + \zeta^{-1})\|b\|^2$. This concludes the proof.
\end{proof}

The compensating term $- p\xi \| w_k - \bar{x}_{k+1} \|^2$ comes at the cost of two additional terms. By comparing equations~\eqref{eq:lsvrp_lyap_first} and~\eqref{eq:lsvrp_lyap_second}, we observe that many terms cancel each other out. This cancellation is the main idea underlying Theorem~\ref{thm:lsvrp_tech_conv}.

\subsection{Combined Lyapunov Analysis}

With the bounds established for the individual components of the Lyapunov function, we now combine these results to derive the overall contraction property.

\begin{theorem}\label{thm:lsvrp_tech_conv}
	Let Assumptions~\ref{as:hess_sim_finite_sum} (Hessian similarity) and~\ref{as:strong_cvx} (strong convexity) hold. For a given $\xi \in [0, 1]$ and constants $c > 0$, $\gamma > 0$, and  $\zeta > 0$, suppose the following inequalities are satisfied:
	\begin{gather}
		c \leq \frac{1}{(1 + \mu\gamma)(1 - p(1 - \xi))(1 + \zeta^{-1})}\label{eq:lsvrp_c_advanced},\\
		c p \xi \geq \frac{\delta^2}{\left( \mu + \frac{1}{\gamma}\right)^2(1 - c(1 - p))} \label{eq:lsvrp_stepsize_condition},\\
		1 - r  = (1 - p(1 - \xi))(1 + \zeta) \label{eq:lsvrp_contraction_governance}
	\end{gather}
	where $r$ is a positive constant, and $p \in (0, 1]$ is the probability parameter from Algorithm~\ref{alg:l_svrp}.
	
	Then, the iterates satisfy:
	\begin{equation*}
		\ExpBr{\Lambda_{k+1}} \leq \max\left\{\frac{1}{1 + \mu\gamma}, 1 - r\right\} \ExpBr{\Lambda_k}.
	\end{equation*}
\end{theorem}

The proof is deferred to the appendix.

\subsection{Interpretation of Results}

While \Cref{thm:lsvrp_tech_conv} may initially appear to impose numerous conditions, many of these can be addressed in a straightforward and natural manner. To clarify this, we offer the following corollary, which simplifies and resolves several of these conditions.

\begin{corollary}\label{crl:lsvrp_conv}
	Let~\Cref{as:hess_sim_finite_sum} (Hessian similarity) and ~\Cref{as:strong_cvx} (strong convexity) hold. Further, choose $\gamma > 0$ and $p \in (0, 1]$ such that
	\begin{equation}\label{eq:lsvrp_gamma_cond}
		\left[\frac{2(3 - p)^2}{p^2 (p + 1)} \right]\delta^2 \gamma^2 \leq 1 + \mu\gamma.
	\end{equation}
	Then, for iterates of~\Cref{alg:l_svrp}, it holds that
	\begin{equation}\label{eq:contraction_rate_explicit}
		\ExpBr{\|x_K - x_\ast\|^2} \\
		\leq \max\left\{\frac{1}{1 + \mu\gamma}, 1 - \frac{p}{4}\right\}^K \|x_0 - x_\ast\|^2.
	\end{equation}
\end{corollary}
\begin{proof}
	In the context of~\Cref{thm:lsvrp_tech_conv}, we specify
	\begin{equation}\label{eq:lsvrp_many_defs}
		c = \frac{2p}{(3 - p)^2} \frac{1}{1 + \mu\gamma}, \quad \xi = \frac12, \quad r = \frac{p}{4}.
	\end{equation}
	We now verify that this value of $c$ meets the conditions outlined in~\Cref{thm:lsvrp_tech_conv}.
	The corresponding value of $\zeta$ is obtained via~\Cref{eq:lsvrp_contraction_governance}:
	\begin{eqnarray}
		1 - \frac{p}{4} & \overset{\eqref{eq:lsvrp_contraction_governance}+\eqref{eq:lsvrp_many_defs}}{=} & \left(1 - \frac{p}{2}\right)(1 + \zeta) \notag\\
		&=& 1 - \frac{p}{2} + \zeta \left(1 - \frac{p}{2}\right) \Leftrightarrow\notag\\
		\frac{p}{4} &= &\zeta \left(1 - \frac{p}{2}\right) \Leftrightarrow\notag\\
		\zeta & = & \frac{p}{4 - 2p} \label{eq:lsvrp_zeta_value}.
	\end{eqnarray}
	Drawing from~\eqref{eq:lsvrp_c_advanced}, we obtain the following expression, which simplifies the inequality for easier interpretation:
	\begin{eqnarray*}
		c &\overset{\eqref{eq:lsvrp_c_advanced}}{\leq} &\frac{1}{(1 + \mu\gamma)(1 - p(1 - \xi))(1 + \zeta^{-1})} \\
		& \overset{\eqref{eq:lsvrp_many_defs}}{=} & \frac{1}{(1 + \mu\gamma)\left(1 - \frac{p}{2}\right)(1 + \zeta^{-1})} \\
		& \overset{\eqref{eq:lsvrp_zeta_value}}{=} &\frac{1}{(1 + \mu\gamma)\left(1 - \frac{p}{2}\right)\left(\frac{4}{p} - 1\right)} \\
		& = & \frac{2p}{(1 + \mu\gamma)(p^2 - 6p + 8)}.
	\end{eqnarray*}
	We proceed to verify that 
	\begin{align*}
	c &\overset{\eqref{eq:lsvrp_many_defs}}{=} \frac{2p}{(3 - p)^2} \frac{1}{1 + \mu\gamma} \leq  \frac{2p}{(1 + \mu\gamma)(p^2 - 6p + 8)}.
	\end{align*}
	Given that $p \in [0, 1]$, the following chain of inequalities holds:
	\begin{equation}\label{eq:lsvrp_tech_bound_1_p}
	\begin{aligned}
		c &\overset{\eqref{eq:lsvrp_many_defs}}{=}  \frac{2p}{(3 - p)^2} \frac{1}{1 + \mu\gamma} \\
		&\leq  \frac{2}{(3 - p)^2}\frac{1}{1 + \mu\gamma} \leq  \frac{2}{4}\frac{1}{1 + \mu\gamma} \leq \frac12.
	\end{aligned}
	\end{equation}
	The final step is to confirm that the stepsize condition given in~\eqref{eq:lsvrp_gamma_cond}  is sufficient to satisfy the remaining condition~\eqref{eq:lsvrp_stepsize_condition}:
	\begin{eqnarray*}
		& & \frac{\delta^2}{\left( \mu + \frac{1}{\gamma}\right)^2} \overset{\eqref{eq:lsvrp_stepsize_condition}}{\leq} \frac12 c p (1 - c(1 - p)) \\
		& \overset{\eqref{eq:lsvrp_tech_bound_1_p}}{\Leftarrow}&\frac{\delta^2}{\left( \mu + \frac{1}{\gamma}\right)^2} \leq \frac14 c p (p+1) \\
		& \overset{\eqref{eq:lsvrp_many_defs}}{\Leftrightarrow}& \frac{\delta^2\gamma^2}{\left( 1 + \mu\gamma\right)^2} \leq \frac{p^2 (p + 1)}{2(3 - p)^2} \frac{1}{1 + \mu\gamma}\\
		& \Leftrightarrow & \delta^2 \gamma^2 \leq \left[\frac{p^2 (p + 1)}{2(3 - p)^2} \right] (1 + \mu\gamma) \\
		& \Leftrightarrow & \left[\frac{2(3 - p)^2}{p^2 (p + 1)} \right]\delta^2 \gamma^2 \leq 1 + \mu\gamma.
	\end{eqnarray*}
	Unrolling the contraction recurrence, we finish the proof.
\end{proof}

We also note a similar result can be obtained for convex functions.

\begin{corollary}\label{crl:lsvrp_conv_cvx}
	Let~\Cref{as:hess_sim_finite_sum} (Hessian similarity) and~\Cref{as:strong_cvx} with $\mu = 0$ (convexity) hold. Let us denote $x_\ast \in \Argmin f$ and $f^\ast \eqdef \min\limits_{x\in\RR^d} f(x)$. Choose $\gamma > 0$ and $p \in (0, 1]$ such that
	\begin{equation}\label{eq:lsvrp_gamma_cond_copy}
		\left[\frac{2(3 - p)^2}{p^2 (p + 1)} \right]\delta^2 \gamma^2 \leq 1 + \mu\gamma.
	\end{equation}
	Then, for iterates of~\Cref{alg:l_svrp}, it holds that
	\begin{equation}
		\frac{1}{K}\sum\limits_{k=1}^{K} \ExpBr{{f(\bar{x}_{k})}} - f^\ast \leq \frac{\|x_0 -x_\ast\|^2}{2 \gamma K}.
	\end{equation}
\end{corollary}

The final step involves specifying an appropriate stepsize \(\gamma\) and deriving the corresponding iteration complexity estimate. This is addressed in~\Cref{crl:lsvrp_conv2}, with the proof deferred to the appendix.

\begin{corollary}\label{crl:lsvrp_conv2}
   Let Assumptions~\ref{as:hess_sim_finite_sum} (Hessian similarity) and ~\ref{as:strong_cvx} (strong convexity) hold. Define the condition number $\kappa$ as
    $
    \kappa \eqdef \frac{\delta}{\mu}.
    $
     If the stepsize is chosen as
	\begin{equation}\label{eq:stepsize_theory}
	\begin{aligned}
	\gamma &= \frac{1}{\delta} \frac{p}{3-p}\sqrt{\frac{p+1}{2}} = \Theta\left(\frac{p}{\delta}\right),
	\end{aligned}
	\end{equation}
	and the number of iterations of~\Cref{alg:l_svrp} satisfies
	\begin{align*}
		K &\geq \left(1 + \frac{1}{p}\left(\frac{3\delta}{\mu}  + 4 \right)\right) \log\left(\frac{\|x_0 - x_\ast\|^2}{\varepsilon} \right) \\
		&= \widetilde{\cO}\left(\frac{\kappa + 1}{p}\right),
	\end{align*}
	then the following bound holds:
	$$
	\ExpBr{\|x_K - x_\ast\|^2} \leq \varepsilon.
	$$
\end{corollary}

\begin{remark}
When $p = \frac{1}{n}$, we obtain the iteration complexity
$$
O\left((n \kappa + n) \log{\frac{\|x_0 - x_\ast\|^2}{\varepsilon}}\right).
$$
Compared to the original complexity $\widetilde{\cO}\left(\kappa^2 + n\right)$ by~\citet{khaled2023faster}, our result is better when $\kappa \geq n$.
\end{remark}
\begin{remark}
As noted by \citet{khaled2023faster}, in the communication model discussed in their paper, the iteration complexity and communication complexity of \algname{L-SVRP} are asymptotically the same.
\end{remark}

\section{EXPERIMENTS}\label{sec:exps}
To evaluate the tightness of our theoretical convergence rate, we conduct experiments across a grid of 48 parameter configurations. Specifically, we vary the number of functions $n$ among the set $\{10, 25, 50, 100, 250, 500\}$ and adjust the Hessian heterogeneity parameter $\delta^2$ over eight values exponentially increasing from $10^0$ to $10^{7}$. Due to the stochastic nature of the data generation process, we do not obtain fixed values of $\delta^2$ for each setting; however, the generated values are approximately similar across runs.

For our experiments, we utilize an internal computational cluster node equipped with 10 CPUs. The code is implemented in Python and leverages standard libraries, including Numpy~\citep{harris2020array} and Scipy~\citep{scipy}.

We perform our tests on quadratic functions defined as follows:

\begin{equation}\label{eq:quad_f}
	f_i(x) = \frac{1}{2} x^\top \bA_i x + b_i^\top x,
\end{equation}

where each $\bA_i \in \mathbb{S}_{++}^d$ is a symmetric positive definite matrix of dimension $d$, and $b_i \in \mathbb{R}^d$ is a vector of the same dimension. In all experiments, we set $d = 100$. For the sake of simplicity, we set \(\avein b_i = 0\). It is important to note, however, that this simplification prevents the interpolation regime from being satisfied.

By \Cref{as:hess_sim_finite_sum}, if the parameter $\delta^2$ exists, it can be explicitly expressed as follows:

\begin{equation*}
\delta^2 =\sup_{\substack{x, y \in \mathbb{R}^d \\ x \neq y}} \frac{\sum\limits_{i=1}^n \| \nabla f_i(x) - \nabla f_i(y) - \nabla f(x) + \nabla f(y)\|^2}{\|x - y\|^2}.
\end{equation*}

We demonstrate that, for the quadratic problem~\eqref{eq:quad_f}, this quantity is finite:

\begin{align*}
	&\sup_{\substack{x, y \in \mathbb{R}^d \\ x \neq y}} \frac{\avein \| \nabla f_i(x) - \nabla f_i(y) - \nabla f(x) + \nabla f(y)\|^2}{\|x - y\|^2} \\
	& =  \sup_{\substack{x, y \in \mathbb{R}^d \\ x \neq y}}\frac{\avein \|\bA_i (x - y) - \bar{\bA} (x - y)\|^2}{\|x - y\|^2} \\
	& =  \sup_{\substack{x, y \in \mathbb{R}^d \\ x \neq y}} \frac{\avein \|(\bA_i - \bar{\bA}) (x - y)\|^2}{\|x - y\|^2} \\ 
	& =  \sup\limits_{\substack{z\in\RR^d \\ z \neq 0 }} \frac{\avein \|(\bA_i - \bar{\bA}) z\|^2}{\|z\|^2}\\
	& = \sup\limits_{\substack{z\in\RR^d \\ z \neq 0 }} \frac{z^\top \left[\avein (\bA_i - \bar{\bA})^\top (\bA_i - \bar{\bA}) \right] z}{\|z\|^2},
\end{align*}
where, in the second equality, we omit the terms \(b_i\) and \(b\) as they cancel out, thereby saving space in the paper. 

Recalling that all matrices under consideration are symmetric, we have:

\begin{align*}
	\delta^2 & = \sup\limits_{x\in\RR^d} \frac{z^\top \left[\avein (\bA_i - \bar{\bA})^\top (\bA_i - \bar{\bA}) \right] z}{\|z\|^2}\\
	& =  \sup\limits_{x\in\RR^d} \frac{z^\top \left[\avein (\bA_i - \bar{\bA}) (\bA_i - \bar{\bA}) \right] z}{\|z\|^2} \\
& =  \sup\limits_{x\in\RR^d} \frac{z^\top \left[\avein \bA_i^2 - \bar{\bA}^2 \right] z}{\|z\|^2} \\
& =  \left\| \avein \bA_i^2 - \bar{\bA}^2\right\|,
\end{align*}
where $\bar{\bA}$ represents the average matrix $\frac{1}{n} \sum_{i=1}^n \bA_i$, and the final equality follows from the definition of the matrix spectral norm. Consequently, for quadratic problems, the parameter 
$$
\delta^2~= ~\left\| \avein \bA_i^2 - \bar{\bA}^2\right\|
$$
 is finite and well-defined. This result was previously established in~\citep{szlendak2022permutation}.

To generate the matrices \( \bA_i \), we create random orthogonal matrices \( \bQ_i \) and random diagonal matrices whose diagonal entries are uniformly sampled from the interval \( [1, s] \). The parameter \( s \) varies over the set $\{5, 10, 50, 100, 500, 1000, 5000, 10000\}$. This generation process ensures that the strong convexity assumption holds with a constant $ \mu =1 $, while allowing $\delta^2$ to range from relatively small to large values.

\begin{figure}
	\includegraphics[width=\linewidth]{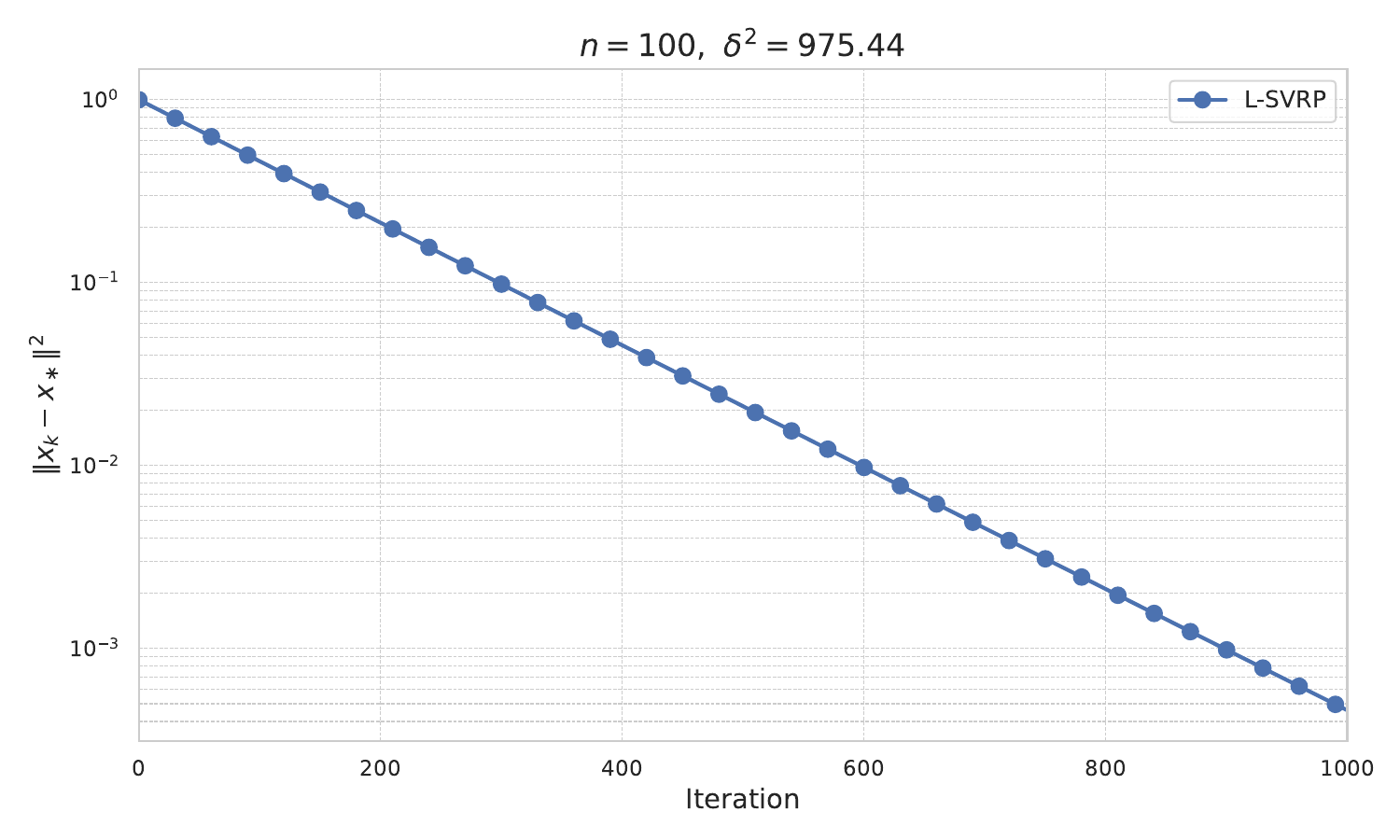}
	\caption{Convergence of \algname{L-SVRP} under one of the 48 configurations examined, where the number of functions $n = 100$ and the Hessian similarity parameter $\delta^2 = 975.44$. As predicted by theoretical analysis, the convergence is linear. We aggregate the slopes from such figures to produce~\Cref{fig:comparing_rates}.}
	\label{fig:example_one_run}
\end{figure}

\begin{figure}
	\centering
	\includegraphics[width=\linewidth]{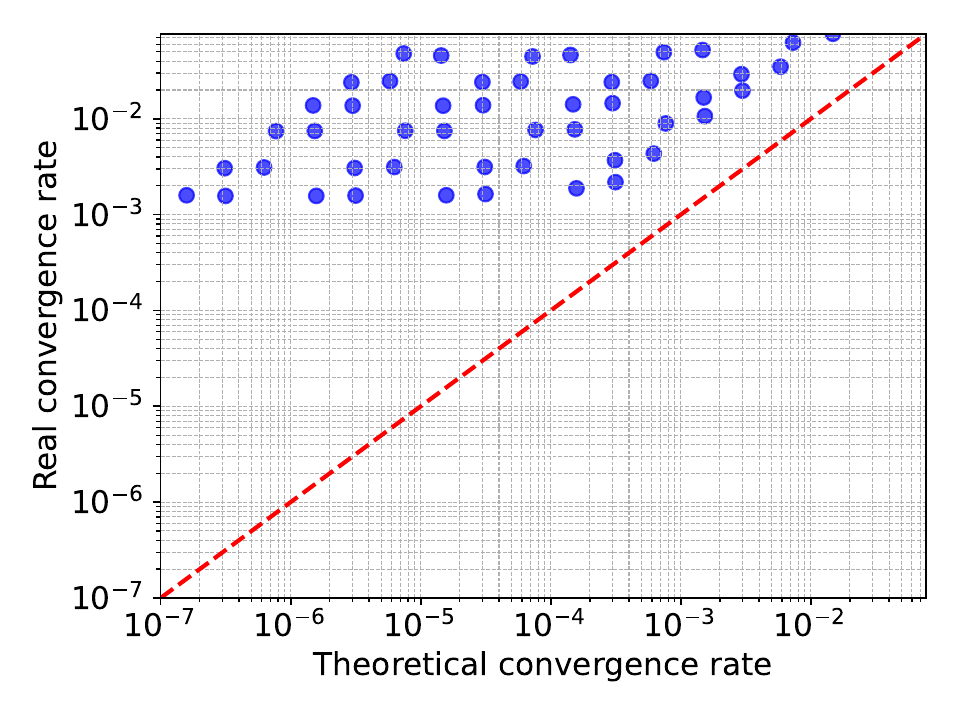}
	\caption{Comparison of real convergence rate vs. theoretical convergence rate. Each point represents one of the 48 problem instances evaluated across varying values of \( n \) and \( \delta^2 \). The red dashed line denotes the identity line \( y = x \). The fact that all points lie above this line confirms the validity of our theoretical analysis. However, the empirical convergence rates are often significantly better than the theoretical predictions, indicating that the theoretical bounds may be further tightened.}
	\label{fig:comparing_rates}
\end{figure}

For each configuration, we run the \algname{L-SVRP} algorithm for 1,000 iterations, repeating this process 200 times to average the squared distances $\| x_k - x_\ast \|^2$, thereby approximating the expectation $\ExpBr{ \| x_k - x_\ast \|^2 }$. The stepsize is defined as in Equation~\eqref{eq:stepsize_theory}. As detailed in the appendix and illustrated in \Cref{fig:example_one_run}, the dynamics of $\log \| x_k - x_\ast \|^2 $ exhibit linear behavior, which is consistent with our theoretical predictions. We consider $1 - x$, where $x$ is the slope of the aforementioned plot, as the empirical convergence rate and compare this value against the theoretical rate, reported in~\Cref{eq:contraction_rate_explicit}. Our results are presented in Figure~\ref{fig:comparing_rates}. As observed, the empirical convergence closely aligns with the theoretical predictions. However, since the real convergence rate is often significantly better than the theoretical estimate, this suggests that there may be potential for further improvements in the analysis.

\clearpage
\bibliography{bibliography}
\clearpage

\section*{Checklist}
\begin{enumerate}

	\item For all models and algorithms presented, check if you include:
	\begin{enumerate}
		\item A clear description of the mathematical setting, assumptions, algorithm, and/or model. Yes
		\item An analysis of the properties and complexity (time, space, sample size) of any algorithm. Yes
		\item (Optional) Anonymized source code, with specification of all dependencies, including external libraries.
	\end{enumerate}

	\item For any theoretical claim, check if you include:
	\begin{enumerate}
		\item Statements of the full set of assumptions of all theoretical results. Yes
		\item Complete proofs of all theoretical results. Yes
		\item Clear explanations of any assumptions. Yes
	\end{enumerate}

	\item For all figures and tables that present empirical results, check if you include:
	\begin{enumerate}
		\item The code, data, and instructions needed to reproduce the main experimental results (either in the supplemental material or as a URL). Yes
		\item All the training details (e.g., data splits, hyperparameters, how they were chosen). Yes
		\item A clear definition of the specific measure or statistics and error bars (e.g., with respect to the random seed after running experiments multiple times). Yes
		\item A description of the computing infrastructure used. (e.g., type of GPUs, internal cluster, or cloud provider). Yes
	\end{enumerate}
	
	\item If you are using existing assets (e.g., code, data, models) or curating/releasing new assets, check if you include:
	\begin{enumerate}
		\item Citations of the creator If your work uses existing assets. Yes
		\item The license information of the assets, if applicable. Not Applicable
		\item New assets either in the supplemental material or as a URL, if applicable. Not Applicable
		\item Information about consent from data providers/curators. Not Applicable
		\item Discussion of sensible content if applicable, e.g., personally identifiable information or offensive content. Not Applicable
	\end{enumerate}
	
	\item If you used crowdsourcing or conducted research with human subjects, check if you include:
	\begin{enumerate}
		\item The full text of instructions given to participants and screenshots. Not Applicable
		\item Descriptions of potential participant risks, with links to Institutional Review Board (IRB) approvals if applicable. Not Applicable
		\item The estimated hourly wage paid to participants and the total amount spent on participant compensation. Not Applicable
	\end{enumerate}
	
\end{enumerate}

\clearpage
\appendix
\onecolumn
\section{Auxiliary facts}

For ease of reference, we present the following standard lemmas.

\begin{lemma}
Let $\{a_i\}_{i=1}^n$ be a set of arbitrary vectors in $\RR^d$. Let the average vector be denoted as $\bar{a} = \avein a_i$. Then, the following identity holds: 
\begin{equation}\label{eq:bias_variance_decomposition}
\avein \|a_i - \bar{a}\|^2 = \avein \|a_i\|^2 - \|\bar{a}\|^2.
\end{equation}
\end{lemma}
\begin{proof}
The proof is straight-forward.
\begin{eqnarray*}
\avein \|a_i - \bar{a}\|^2 & = & \avein \left[\|a_i\|^2 + \|\bar{a}\|^2 - 2\lin{a_i, \bar{a}}\right]\\
& = & \avein \|a_i\|^2 + \|\bar{a}\|^2 - 2\avein \lin{a_i, \bar{a}}\\
& = & \avein \|a_i\|^2 + \|\bar{a}\|^2 - 2 \lin{\avein a_i, \bar{a}}\\
& = & \avein \|a_i\|^2 + \|\bar{a}\|^2 - 2 \|\bar{a}\|^2\\
& = & \avein \|a_i\|^2 - \|\bar{a}\|^2.
\end{eqnarray*}
\end{proof}

\begin{lemma}\label{lm:scalar_prod_bound}
	For any vectors $a, b \in\RR^d$ and a positive scalar $\zeta > 0$, it holds that
	\begin{equation}\label{eq:scalar_prod_bound}
		\lin{a, b} - \frac{1}{2\zeta}\|a\|^2\leq \frac{\zeta}{2}\|b\|^2.
	\end{equation}
\end{lemma}
\begin{proof} Expanding the quadratics, we get:
\begin{eqnarray*}
0 \leq \frac12\left\|\sqrt{\zeta} b - \frac{1}{\sqrt{\zeta}}a \right\|^2 = \frac12 \left[\zeta \|b\|^2 + \frac{1}{\zeta}\|a\|^2 - 2\lin{a, b} \right].
\end{eqnarray*}
It remains to rearrange the terms to finish the proof.
\end{proof}

We also present key properties of proximity operators.

\begin{definition}
	Let $f:\RR^d \rightarrow \RR \cup \{+\infty\}$ be a proper, closed, convex function and $\gamma > 0$ be a positive scalar. Then, the proximal operator with respect to $f$ is defined as follows:
	\begin{equation}\label{eq:prox_definition}
		\prox_{\gamma f}(x) = \argmin\limits_{y \in \RR^d} \left\{\frac{1}{2\gamma}\|x - y\|^2 + f(y) \right\}.
	\end{equation}
\end{definition}

\begin{lemma}[Proposition 16.44, ~\citep{bauschke2011convex}]\label{lm:prox_opt_condition}
	Let $f:\RR^d \rightarrow  \RR \cup \{+\infty\}$ be a proper, closed, convex function and $\gamma > 0$ be a positive scalar. Then,
	\begin{equation}\label{eq:prox_opt_condition}
		\prox_{\gamma f}(x) = y \Leftrightarrow x -y \in \gamma \partial f(y).
	\end{equation}
\end{lemma}

\begin{lemma}\label{lm:prox_contractivity}
	Let $f: \mathbb{R}^d \to \mathbb{R}$ be a $\mu$-strongly convex function  (\Cref{as:strong_cvx}). Then, for any $x, y \in \mathbb{R}^d$ and any $\gamma > 0$, the proximity operator satisfies
	\begin{equation}\label{eq:prox_contractivity}
		\left\| \prox_{\gamma f}(x) - \prox_{\gamma f}(y) \right\|^2 \leq \frac{1}{(1 + \gamma \mu)^2} \| x - y \|^2.
	\end{equation}
\end{lemma}

\begin{proof}
	Let us denote $ x' = \prox_{\gamma f}(x)$ and $y' = \prox_{\gamma f}(y)$. By the optimality condition of the proximal operator, we have
	\begin{equation*}
		x - x' = \gamma\nabla f( x' ), \quad  y - y' = \gamma\nabla f( y' ).
	\end{equation*}
	Subtracting these two equations yields
	\[
	x - y = x' - y' + \gamma \big( \nabla f( x' ) - \nabla f( y' ) \big).
	\]
	Taking the inner product with $x' - y'$, we obtain
	\[
	\langle x - y, x' - y' \rangle = \| x' - y' \|^2 + \gamma \langle \nabla f( x' ) - \nabla f( y' ), x' - y' \rangle.
	\]
	Using the strong convexity of $f$, we have
	\begin{equation}\label{eq:strong_convexity_scalar}
		\langle \nabla f( x' ) - \nabla f( y' ), x' - y' \rangle \geq \mu \| x' - y' \|^2.
	\end{equation}
	Substituting inequality~\eqref{eq:strong_convexity_scalar} into the previous expression, we get
	\[
	\langle x - y, x' - y' \rangle \geq \left( 1 + \gamma \mu \right) \| x' - y' \|^2.
	\]
	By the Cauchy–Schwarz inequality,
	\[
	\| x - y \| \| x' - y' \| \geq \langle x - y, x' - y' \rangle.
	\]
	Combining the two inequalities, we obtain
	\[
	\| x - y \| \| x' - y' \| \geq \left( 1 + \gamma \mu \right) \| x' - y' \|^2,
	\]
	which simplifies to
	\[
	\| x' - y' \| \leq \frac{1}{1 + \gamma \mu} \| x - y \|.
	\]
	This completes the proof.
\end{proof}

\section{Proof of~\Cref{lm:lsvrp_lyap_first}}

We begin by introducing the following technical lemma, which follows the proof strategy of the \algname{DANE} algorithm as described in~\citep{stich_dane}.

\begin{lemma}\label{lm:lsvrp_technical_lemma}
	Let~\Cref{as:strong_cvx} (strong convexity) hold. Let us define
	\begin{equation}\label{eq:bar_x_def_aux}
	\bar{x}_{k+1} \eqdef \ExpBr{x_{k+1} | x_k, w_k}
	\end{equation} and
	\begin{equation}\label{eq:psi_def_aux}
	\psi_i (x) \eqdef f_i(x) - f(x).
	\end{equation}
	Then, after one iteration of \Cref{alg:l_svrp}, for any $y \in \RR^d$, the following inequality holds:
	\begin{eqnarray*}
		\ExpBr{\|x_{k+1} - x_\ast\|^2 | x_k, w_k} \leq \frac{1}{1 + \mu\gamma} \|x_k - x_\ast\|^2 & \text{(contraction)} &\\
		- \frac{2}{ \mu + \frac{1}{\gamma}}[f(y) - f(x_\ast)] & (\text{always} \leq 0)&\\
		- \frac{2}{\mu + \frac{1}{\gamma}}  \ExpBr{\lin{\nabla \psi_{i_k}(y) - \nabla \psi_{i_k}(w_k), x_{k+1} - y} + \frac{\mu + \frac{1}{\gamma}}{2}\|x_{k+1} - y\|^2| x_k, w_k} &(\text{bounded by} \leq \frac{\delta^2 \|y- w_k\|^2}{\left(\mu + \frac{1}{\gamma} \right)^2})&\\
		- \frac{1}{1 + \mu\gamma}\|y - x_k\|^2 & (\text{negative term})& \\
		- \frac{2}{1 + \mu\gamma} \lin{\bar{x}_{k+1} - y, y + \gamma \nabla f(y) - x_k} &(\text{remaining term})&.
	\end{eqnarray*}
\end{lemma}
\begin{remark}
To eliminate the "remaining" term, represented by the scalar product, it is necessary to choose $y \in \RR^d$ such that $ \langle \bar{x}_{k+1} - y, y + \gamma \nabla f(y) - x_k \rangle \geq 0$. This condition is satisfied, for example, when $y = \prox_{\gamma f}(x_k)$ or $y = \bar{x}_{k+1}$.
\end{remark}
\begin{proof}
	In this proof, all expectations are taken with respect to $x_k$ and $w_k$. Let us denote $h_k  = \nabla f_{i_k}(w_k) - \nabla f(w_k)$. We begin as follows:
	\begin{eqnarray}
		x_{k+1} &\overset{\text{Alg.}~\ref{alg:l_svrp} }{=}& \prox_{\gamma f_{i_k}}(x_k + \gamma h_k)\notag\\
		& \overset{\eqref{eq:prox_definition}}{=} & \argmin\limits_{x\in\RR^d}\left\{f_{\xi_k}(x) + \frac{1}{2\gamma}\|x - x_k - \gamma h_k\|^2 \right\}\notag\\
		& = & \argmin\limits_{x\in\RR^d}\left\{f_{\xi_k}(x) + \frac{1}{2\gamma}\left[\|x - x_k\|^2  + \gamma^2\|h_k\|^2 - 2\gamma\lin{x - x_k, h_k} \right] \right\} \label{eq:lsvrp_x_k_1_simplification}\\
		& = & \argmin\limits_{x\in\RR^d}\left\{f_{\xi_k}(x) + \frac{1}{2\gamma}\|x - x_k\|^2 - \lin{x, h_k}  + \frac{\gamma}{2}\|h_k\|^2 + \lin{x_k, h_k}  \right\} \notag\\
		& = & \argmin\limits_{x \in \RR^d} \left\{f_{\xi_k}(x) - \lin{x, h_k} + \frac{1}{2\gamma}\|x - x_k\|^2 \right\}\notag.
	\end{eqnarray}
	For brevity, we denote $f_{i_k}^k (x)\eqdef f_{i_k}(x) - \lin{x, \nabla f_{i_k}(w_k) - \nabla f(w_k)} + \frac{1}{2\gamma}\|x - x_k\|^2$. Then, according to~\eqref{eq:lsvrp_x_k_1_simplification}, 
	\begin{equation}\label{eq:lsvrp_x_k_1_simplification_2}
		x_{k+1} = \argmin f_{i_k}^k(x).
	\end{equation}
	
	By leveraging the $\mu + \frac{1}{\gamma}$-strong convexity of $f_{i_k}^k$ and applying~\eqref{eq:lsvrp_x_k_1_simplification_2}, we derive the following for any $x \in \RR^d$:
	\begin{eqnarray*}
		f_{i_k}^k(x) \geq f_{i_k}^k(x_{k+1}) + \frac{\mu + \frac{1}{\gamma}}{2} \|x - x_{k+1}\|^2.
	\end{eqnarray*}
	Expanding $f_{i_k}^k$, we obtain
	\begin{eqnarray*}
		&&f_{i_k}(x) - \lin{x, \nabla f_{i_k}(w_k) - \nabla f(w_k)} + \frac{1}{2\gamma}\|x - x_k\|^2 \\
		&&\geq  f_{i_k}(x_{k+1}) - \lin{x_{k+1}, \nabla f_{i_k}(w_k) - \nabla f(w_k)} + \frac{1}{2\gamma}\|x_{k+1} - x_k\|^2 +  \frac{\mu + \frac{1}{\gamma}}{2} \|x - x_{k+1}\|^2.
	\end{eqnarray*}
	Exploiting the strong convexity of $f_{i_k}$, we further obtain the following for any arbitrary $y\in\RR^d$:
	\begin{eqnarray*}
		&&f_{i_k}(x) - \lin{x, \nabla f_{i_k}(w_k) - \nabla f(w_k)} + \frac{1}{2\gamma}\|x - x_k\|^2 \\
		&&\geq  f_{i_k}(y) + \lin{\nabla f_{i_k}(y), x_{k+1} - y} + \frac{\mu}{2}\|y - x_{k+1}\|^2\\
		&& - \lin{x_{k+1}, \nabla f_{i_k}(w_k) - \nabla f(w_k)} + \frac{1}{2\gamma}\|x_{k+1} - x_k\|^2 +  \frac{\mu + \frac{1}{\gamma}}{2} \|x - x_{k+1}\|^2.
	\end{eqnarray*}
	
	By taking the conditional expectation $\ExpBr{\cdot | x_k, w_k}$, we get
	\begin{eqnarray}
		&&f(x) + \frac{1}{2\gamma} \|x - x_k\|^2 \geq f(y) + \frac{\mu + \frac{1}{\gamma}}{2} \ExpBr{\|x - x_{k+1}\|^2} \label{eq:lsvrp_aux1}\\
		&&+ \ExpBr{\lin{\nabla f_{i_k}(y), x_{k+1} - y} - \lin{x_{k+1}, \nabla f_{i_k}(w_k) - \nabla f(w_k)}} + \ExpBr{\frac{\mu}{2}\|x_{k+1} - y\|^2 + \frac{1}{2\gamma}\|x_{k+1} - x_k\|^2}.\notag
	\end{eqnarray}
	Recall that $\psi_{i_k} \eqdef f_{i_k} - f$. Given that $\ExpBr{\lin{y, \nabla f_{i_k}(w_k) - \nabla f(w_k)}| x_k, w_k} = 0$, we additionally observe
	\begin{eqnarray*}
		&&\ExpBr{\lin{\nabla f_{i_k}(y), x_{k+1} - y} - \lin{x_{k+1}, \nabla f_{i_k}(w_k) - \nabla f(w_k)}} \\
		&& = \ExpBr{\lin{\nabla f_{i_k}(y) - \nabla f(y) - [\nabla f_{i_k}(w_k) - \nabla f(w_k)], x_{k+1} - y}} + \ExpBr{\lin{\nabla f(y), x_{k+1} - y}} \\
		&&= \ExpBr{\lin{\nabla f_{i_k}(y) - \nabla f(y) - [\nabla f_{i_k}(w_k) - \nabla f(w_k)], x_{k+1} - y}} + \lin{\nabla f(y), \bar{x}_{k+1} - y} \\
		&& = \ExpBr{\lin{\nabla \psi_{i_k}(y) - \nabla \psi_{i_k}(w_k), x_{k+1} - y}} + \lin{\nabla f(y), \bar{x}_{k+1} - y}.
	\end{eqnarray*}
	It is also evident that
	\begin{eqnarray*}
		&&\ExpBr{\frac{\mu}{2}\|x_{k+1} - y\|^2 + \frac{1}{2\gamma}\|x_{k+1} - x_k\|^2} \\
		&& = \ExpBr{\frac{\mu}{2}\|x_{k+1} - y\|^2 + \frac{1}{2\gamma}\|x_{k+1} - y\|^2 + \frac{1}{2\gamma}\|y - x_k\|^2 + \frac{1}{\gamma}\lin{x_{k+1} - y, y - x_k}} \\
		&& = \ExpBr{\frac{\mu + \frac{1}{\gamma}}{2}\|x_{k+1} - y\|^2 } + \frac{1}{2\gamma} \|y - x_k\|^2 + \frac{1}{\gamma}\lin{\bar{x}_{k+1} - y, y - x_k}.
	\end{eqnarray*}
	Plugging last two equations into~\eqref{eq:lsvrp_aux1}, we obtain
	\begin{eqnarray*}
		&&f(x) + \frac{1}{2\gamma} \|x - x_k\|^2 \geq f(y) + \frac{\mu + \frac{1}{\gamma}}{2} \ExpBr{\|x - x_{k+1}\|^2}\\
		&&+\ExpBr{\lin{\nabla \psi_{i_k}(y) - \nabla \psi_{i_k}(w_k), x_{k+1} - y}} + \lin{\nabla f(y), \bar{x}_{k+1} - y} \\
		&& + \ExpBr{\frac{\mu + \frac{1}{\gamma}}{2}\|x_{k+1} - y\|^2 } + \frac{1}{2\gamma} \|y - x_k\|^2 + \frac{1}{\gamma}\lin{\bar{x}_{k+1} - y, y - x_k}.
	\end{eqnarray*}
	Rearranging terms, we get
	\begin{eqnarray*}
		&&f(x) + \frac{1}{2\gamma} \|x - x_k\|^2 \geq f(y) + \frac{\mu + \frac{1}{\gamma}}{2} \ExpBr{\|x - x_{k+1}\|^2} + \frac{1}{\gamma} \lin{\bar{x}_{k+1} - y, y + \gamma \nabla f(y) - x_k} \\
		&& + \ExpBr{\lin{\nabla \psi_{i_k}(y) - \nabla \psi_{i_k}(w_k), x_{k+1} - y}} + \frac{\mu + \frac{1}{\gamma}}{2}\ExpBr{\|x_{k+1} - y\|^2 } + \frac{1}{2\gamma} \|y - x_k\|^2.
	\end{eqnarray*}
	Rearranging the terms once again and setting $x = x_\ast$, we have:
	\begin{eqnarray*}
		\ExpBr{\|x_{k+1} - x_\ast\|^2} \leq \frac{1}{1 + \mu\gamma} \|x_k - x_\ast\|^2 & \text{(contraction)} &\\
		- \frac{2}{ \mu + \frac{1}{\gamma}}[f(y) - f(x_\ast)] & (\text{always} \leq 0)&\\
		- \frac{2}{\mu + \frac{1}{\gamma}}  \ExpBr{\lin{\nabla \psi_{i_k}(y) - \nabla \psi_{i_k}(w_k), x_{k+1} - y} + \frac{\mu + \frac{1}{\gamma}}{2}\|x_{k+1} - y\|^2} &(\text{can be bounded by} \leq \frac{\delta^2 \|y- w_k\|^2}{\left(\mu + \frac{1}{\gamma} \right)^2})&\\
		- \frac{1}{1 + \mu\gamma}\|y - x_k\|^2 - \frac{2}{1 + \mu\gamma} \lin{\bar{x}_{k+1} - y, y + \gamma \nabla f(y) - x_k} &(\text{remaining terms})&.
	\end{eqnarray*}
\end{proof}

The proof of~\Cref{lm:lsvrp_lyap_first} now follows as a corollary of~\Cref{lm:lsvrp_technical_lemma}.
\paragraph{Proof of~\Cref{lm:lsvrp_lyap_first}.}
	If we substitute $y = \bar{x}_{k+1}$ into~\Cref{lm:lsvrp_technical_lemma}, the last term, $\lin{\bar{x}_{k+1} - y, y + \gamma \nabla f(y) - x_k}|_{y = \bar{x}_{k+1}} = 0$. Furthermore, by noting that $f(\bar{x}_{k+1}) - f(x_\ast) \geq 0$, the statement can be established with the application of~\Cref{lm:lsvrp_technical_lemma}.
	
\section{Proof of~\Cref{thm:lsvrp_tech_conv}}

\begin{proof}
	Recall that the Lyapunov sequence is defined as follows:
	\begin{equation}\label{eq:lsvrp_lyapunov_func_def}
	\Lambda_k = \|x_k - x_\ast\|^2 + c\|w_k - x_k\|^2.
	\end{equation}
	Combining~\Cref{lm:lsvrp_lyap_first} and~\Cref{lm:lsvrp_lyap_second}, we obtain:
	\allowdisplaybreaks{
		\begin{eqnarray}
			\ExpBr{\Lambda_{k+1} | x_k, w_k} &\overset{\eqref{eq:lsvrp_lyapunov_func_def}}{=}& \ExpBr{\|x_{k+1} - x_\ast\|^2 | x_k, w_k} + c \ExpBr{\|w_{k+1} - x_{k+1}\|^2 | x_k, w_k}  \notag\\
			&\overset{\eqref{eq:lsvrp_lyap_first} + \eqref{eq:lsvrp_lyap_second}}{\leq}& \frac{1}{1 + \mu\gamma} \|x_k - x_\ast\|^2 \notag\\
			& &- \frac{2}{\mu + \frac{1}{\gamma}}  \ExpBr{\lin{\nabla \psi_{i_k}(\bar{x}_{k+1}) - \nabla \psi_{i_k}(w_k), x_{k+1} - \bar{x}_{k+1}} + \frac{\mu + \frac{1}{\gamma}}{2}\|x_{k+1} - \bar{x}_{k+1}\|^2}\notag \\
			& &- \frac{1}{1 + \mu\gamma}\|\bar{x}_{k+1} - x_k\|^2\notag\\
			& & + (1 - p(1 - \xi))(1 + \zeta) c\|w_k - x_k\|^2 \notag\\
			& & + (1 - p(1 - \xi))(1 + \zeta^{-1}) c\|x_k - \bar{x}_{k+1}\|^2 \notag\\
			& &  - cp\xi \|w_k - \bar{x}_{k+1}\|^2 \notag\\
			& & + c(1 - p) \ExpBr{\|\bar{x}_{k+1} - x_{k+1}\|^2 | x_k, w_k}\notag\\
			& = &  \frac{1}{1 + \mu\gamma} \|x_k - x_\ast\|^2 + (1 - p(1 - \xi))(1 + \zeta) c\|w_k - x_k\|^2\label{eq:lsvrp_aux3}\\
			& & - \frac{2}{\mu + \frac{1}{\gamma}}  \ExpBr{\lin{\nabla \psi_{i_k}(\bar{x}_{k+1}) - \nabla \psi_{i_k}(w_k), x_{k+1} - \bar{x}_{k+1}} + \frac{\mu + \frac{1}{\gamma}}{2}\left(1 - c(1 - p)\right)\|x_{k+1} - \bar{x}_{k+1}\|^2}\notag\\
			& & - cp\xi \|w_k - \bar{x}_{k+1}\|^2\notag\\
			& & + \|\bar{x}_{k+1} - x_k\|^2 \left[(1 - p(1 - \xi))(1 + \zeta^{-1}) c - \frac{1}{1 + \mu\gamma}\right]\notag\\
			& \overset{\eqref{eq:scalar_prod_bound}}{\leq} & \max\left\{\frac{1}{1 + \mu\gamma}, 1 - p(1 - \xi)(1 + \zeta)\right\} \Lambda_k \notag\\
			& & + \frac{2}{\mu + \frac{1}{\gamma}} \frac{1}{2\left( \mu + \frac{1}{\gamma}\right)(1 - c(1 - p))} \ExpBr{\|\nabla \psi_{i_k}(\bar{x}_{k+1}) - \nabla \psi_{i_k}(w_k)\|^2} \notag\\
			& & -  cp\xi \|w_k - \bar{x}_{k+1}\|^2, \notag
		\end{eqnarray}
	}
	where we drop the last term since we enforce $c \leq \frac{1}{(1 + \mu\gamma)(1 - p(1 - \xi))(1 + \zeta^{-1})}$. Additionally, observe that
	$$
	c \leq \frac{1}{(1 + \mu\gamma)(1 - p(1 - \xi))(1 + \zeta^{-1})} < \frac{1}{1 - p(1 - \xi)} < \frac{1}{1 - p}.
	$$
	Therefore, $(1 - c(1 - p))\frac{\mu + \frac{1}{\gamma}}{2} > 0$, and~\Cref{lm:scalar_prod_bound} is applicable.
	Then, we continue
	\begin{eqnarray*}
		\ExpBr{\Lambda_{k+1} | x_k, w_k} &\overset{\eqref{eq:lsvrp_aux3}}{\leq} & \max\left\{\frac{1}{1 + \mu\gamma}, (1 - p(1 - \xi))(1 + \zeta)\right\} \Lambda_k \\
		& & + \frac{2}{\mu + \frac{1}{\gamma}} \frac{1}{2\left( \mu + \frac{1}{\gamma}\right)(1 - c(1 - p))} \ExpBr{\|\nabla \psi_{i_k}(\bar{x}_{k+1}) - \nabla \psi_{i_k}(w_k)\|^2} \\
		& & -  cp\xi \|w_k - \bar{x}_{k+1}\|^2 \\
		& \overset{\eqref{eq:hess_sim_finite_sum}}{\leq} & \max\left\{\frac{1}{1 + \mu\gamma}, (1 - p(1 - \xi))(1 + \zeta)\right\} \Lambda_k \\
		& & + \frac{1}{\left( \mu + \frac{1}{\gamma}\right)^2(1 - c(1 - p))} \delta^2 \|w_k - \bar{x}_{k+1}\|^2 \\
		& & -  cp\xi \|w_k - \bar{x}_{k+1}\|^2\\
		& \leq &  \max\left\{\frac{1}{1 + \mu\gamma}, (1 - p(1 - \xi))(1 + \zeta)\right\} \Lambda_k,
	\end{eqnarray*}
	since we enforce $ \frac{\delta^2}{\left( \mu + \frac{1}{\gamma}\right)^2(1 - c(1 - p))} \leq cp\xi$. To complete the proof, it remains to take the full expectation of both sides.
\end{proof}

\section{Proof of~\Cref{crl:lsvrp_conv2}}

\begin{proof}
	Let us define 
	\begin{equation}\label{eq:lsvrp_aux4}
		a \eqdef \frac{2(3 - p)^2}{p^2(p+1)}.
	\end{equation}
	According to Corollary~\ref{crl:lsvrp_conv}, the stepsize $\gamma$ must satisfy the inequality
	$$
	a\delta^2 \gamma^2 - \mu\gamma - 1 \leq 0.
	$$
	
	We verify that the choice $\gamma'= \frac{1}{\delta\sqrt{a}}$ satisfies this condition: $a\delta^2\gamma'^2 - \mu\gamma' - 1 = -\mu \gamma' \leq 0$.
	By~\eqref{eq:contraction_rate_explicit}, we have
	\begin{eqnarray*}
		\ExpBr{\|x_{T} - x_\ast\|^2} \leq \max\left\{1 - \frac{\mu\gamma'}{1 + \mu\gamma'}, 1 - \frac{p}{4}\right\}^T \|x_0 - x_\ast\|^2.
	\end{eqnarray*}
	To achieve $\ExpBr{\|x_{T} - x_\ast\|^2} \leq \varepsilon$, we need to satisfy
	\begin{eqnarray*}
		T &\geq& \left(1 + \frac{1}{\mu\gamma'} + \frac{4}{p} \right) \log\left(\frac{\|x_0 - x_\ast\|^2}{\varepsilon} \right) \\
		&=& \left(1 + \frac{\delta\sqrt{a}}{\mu} + \frac{4}{p} \right) \log\left(\frac{\|x_0 - x_\ast\|^2}{\varepsilon} \right) \\
		& \overset{\eqref{eq:lsvrp_aux4}}{=} & \left(1 + \frac{\delta}{\mu}\frac{3-p}{p} \sqrt{\frac{2}{p+1}} + \frac{4}{p} \right) \log\left(\frac{\|x_0 - x_\ast\|^2}{\varepsilon} \right) \Leftarrow\\
		T &\geq& \left(1 + \frac{\delta}{\mu}\frac{3}{p}  + \frac{4}{p} \right) \log\left(\frac{\|x_0 - x_\ast\|^2}{\varepsilon} \right),
	\end{eqnarray*}
	what completes the proof.
\end{proof}

\newpage

\section{Convex analysis}

The result established for strongly convex functions can be seamlessly extended to the convex case. In summary, all findings presented in the earlier sections remain valid when $\mu = 0$ and for any $ x_\ast \in \Argmin f $ (noting that if $\mu > 0$, the minimizer is unique). However, a key challenge arises: the primary result for strongly convex functions does not yield contraction when $\mu = 0$. Nevertheless, by employing a few modifications to the existing proof, it is possible to derive a new convergence rate.

\begin{lemma}\label{lm:lsvrp_lyap_first_amended}
	Let~\Cref{as:strong_cvx} (strong convexity) hold. Let us define 
	$
	\psi_i(x) \eqdef f_i(x)-f(x),
	$
	and $x_\ast \in \Argmin f$, and $f^\ast = \min\limits_{x \in \RR^d} f(x)$
	Then, for one iteration of Algorithm~\ref{alg:l_svrp}, the following inequality holds:
	\begin{equation}\label{eq:lsvrp_lyap_first_amended}
		\begin{aligned}
			& \ExpBr{\|x_{k+1} - x_\ast\|^2 | x_k, w_k} \leq  \frac{1}{1 + \mu\gamma} \|x_k - x_\ast\|^2 - \frac{2}{\mu + \frac{1}{\gamma}} (f(\bar{x}_{k+1}) - f^\ast) - \ExpBr{\|x_{k+1} - \bar{x}_{k+1}\|^2} \\
			& \quad - \frac{1}{1 + \mu\gamma}\|\bar{x}_{k+1} - x_k\|^2 - \frac{2}{\mu + \frac{1}{\gamma}}  \ExpBr{\lin{\nabla \psi_{i_k}(\bar{x}_{k+1}) - \nabla \psi_{i_k}(w_k), x_{k+1} - \bar{x}_{k+1}} | x_k, w_k}.
		\end{aligned}
	\end{equation}
\end{lemma}
\begin{proof}
Recall that the proof of~\Cref{lm:lsvrp_lyap_first} relies on the technical~\Cref{lm:lsvrp_technical_lemma} by assigning a specific value to $y $ and, more importantly, explicitly eliminating the negative term $ -\frac{2}{\mu + \frac{1}{\gamma}} f(\bar{x}_{k+1} - f^\ast)$. If we retain the functional difference term, we obtain the proof of the current lemma.
\end{proof}

Preserving~\Cref{lm:lsvrp_lyap_second} unchanged, we will reformulate the result of~\Cref{thm:lsvrp_tech_conv} by utilizing~\Cref{lm:lsvrp_lyap_first_amended} in place of~\Cref{lm:lsvrp_lyap_first}, while following the exact same sequence of derivations. This approach yields~\Cref{thm:lsvrp_tech_conv_amended}.

\begin{theorem}\label{thm:lsvrp_tech_conv_amended}
	Let Assumptions~\ref{as:hess_sim_finite_sum} and~\ref{as:strong_cvx} hold. Let $x_\ast \in \Argmin f$ and $f^\ast = \min\limits_{x\in\RR^d} f(x)$. For a given $\xi \in [0, 1]$ and constants $c > 0$, $\gamma > 0$, and  $\zeta > 0$, suppose the following inequalities are satisfied:
	\begin{gather}
		c \leq \frac{1}{(1 + \mu\gamma)(1 - p(1 - \xi))(1 + \zeta^{-1})}\label{eq:lsvrp_c_advanced_amended},\\
		c p \xi \geq \frac{\delta^2}{\left( \mu + \frac{1}{\gamma}\right)^2(1 - c(1 - p))} \label{eq:lsvrp_stepsize_condition_amended},\\
		1 - r  = (1 - p(1 - \xi))(1 + \zeta) \label{eq:lsvrp_contraction_governance_amended}
	\end{gather}
	where $r$ is a positive constant, and $p \in (0, 1]$ is the probability parameter from Algorithm~\ref{alg:l_svrp}.
	
	Define the Lyapunov function:
	\begin{equation*}
		\Lambda_k \eqdef \|x_k - x_\ast\|^2 + c\|w_k - x_k\|^2.
	\end{equation*}
	where $x_k$ and $w_k$ are the iterates of Algorithm~\ref{alg:l_svrp}, and $x_\ast$ is the unique solution to problem~\eqref{eq:finite_sum}. Then, the iterates satisfy:
	\begin{equation}\label{eq:lsvrp_tech_conv_amended}
		\ExpBr{\Lambda_{k+1}} \leq \max\left\{\frac{1}{1 + \mu\gamma}, 1 - r\right\} \ExpBr{\Lambda_k} - \frac{2}{\mu + \frac{1}{\gamma}} (\ExpBr{f(\bar{x}_{k+1})} - f^\ast),
	\end{equation}
\end{theorem}

\begin{proof}
The proof follows the same sequence of derivations as in the proof of~\Cref{thm:lsvrp_tech_conv}, with the sole difference being that~\Cref{lm:lsvrp_lyap_first_amended} is applied instead of~\Cref{lm:lsvrp_lyap_first}.
\end{proof}
	
We are now ready to establish the result for convex functions.

\paragraph{Proof of~\Cref{crl:lsvrp_conv_cvx}}

\begin{proof}
As shown in the proof of~\Cref{crl:lsvrp_conv}, equations~\eqref{eq:lsvrp_many_defs} and~\eqref{eq:lsvrp_gamma_cond} ensure that the following conditions satisfy inequalities~\eqref{eq:lsvrp_c_advanced_amended},~\eqref{eq:lsvrp_stepsize_condition_amended}, and~\eqref{eq:lsvrp_contraction_governance_amended}:  
$$  
c = \frac{2p}{(3 - p)^2} \frac{1}{1 + \mu\gamma}, \quad \xi = \frac12, \quad r = \frac{p}{4},  
$$  
and we replicate the condition on $\gamma$ into~\eqref{eq:lsvrp_gamma_cond_copy}.  

By rearranging the terms in~\eqref{eq:lsvrp_tech_conv_amended}, we obtain  
\begin{equation}\label{eq:lsvrp_cvx_aux1}  
	\frac{2}{\mu + \frac{1}{\gamma}}\left[\ExpBr{f(\bar{x}_{k+1})} - f^\ast\right] \overset{\leq}{\eqref{eq:lsvrp_tech_conv_amended}} \left[\max\left\{\frac{1}{1 + \mu\gamma}, 1 - r\right\} \ExpBr{\Lambda_k} - \ExpBr{\Lambda_{k+1}} \right]  
	\leq \ExpBr{\Lambda_k} - \ExpBr{\Lambda_{k+1}}.  
\end{equation}  

To complete the proof, we sum all inequalities of this form over $k = 0, \dots, K - 1$ and rearrange the terms as follows:  
\begin{eqnarray*}  
	\frac{2}{\mu + \frac{1}{\gamma}}\frac{1}{K} \sum\limits_{k = 1}^{K}  \left[ \ExpBr{f(\bar{x}_k)} - f^\ast\right] \overset{\eqref{eq:lsvrp_cvx_aux1}}{\leq} \frac{1}{K} \sum\limits_{k=0}^{K - 1} \ExpBr{\Lambda_k} - \ExpBr{\Lambda_{k+1}} = \frac{\Lambda_0 - \ExpBr{\Lambda_{K}}}{K} \leq \frac{\Lambda_0}{K} = \frac{\|x_0 - x_\ast\|^2}{K}.  
\end{eqnarray*}
All that remains is to set $\mu = 0$.
\end{proof}

\section{Additional plots}

To empirically confirm the linear rate of convergence across all settings, we present plots similar to~\Cref{fig:example_one_run} for all runs. Figures~\ref{fig:all_plots_1} and~\ref{fig:all_plots_2} display the convergence of~\algname{L-SVRP} for all 48 configurations described in the experimental section, with each figure showing half of the configurations.

\begin{figure}
	\includegraphics[width=\linewidth]{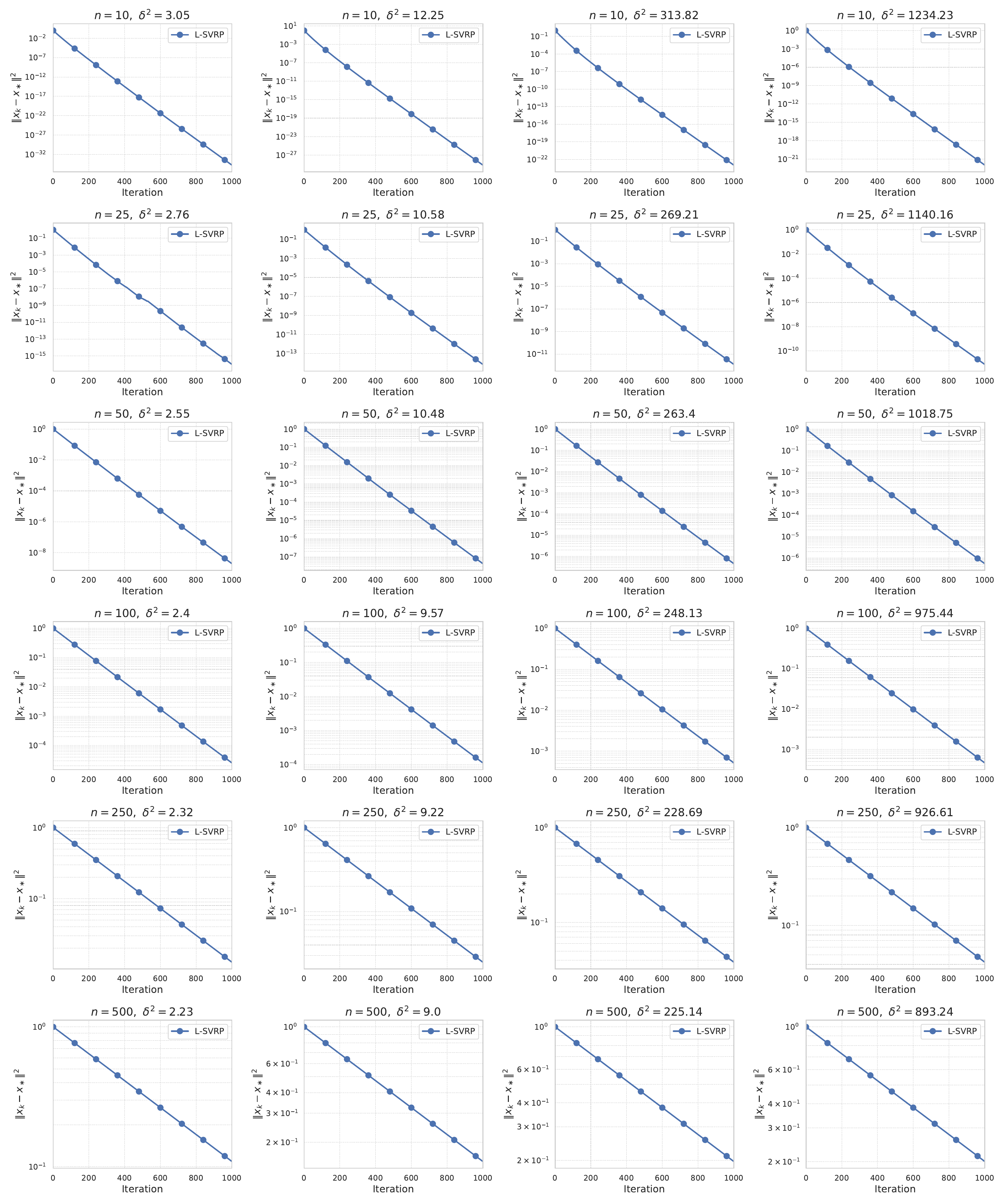}
	\caption{Convergence of~\algname{L-SVRP} for the first half of the 48 configurations described in the experimental section.}
	\label{fig:all_plots_1}
\end{figure}

\begin{figure}
	\includegraphics[width=\linewidth]{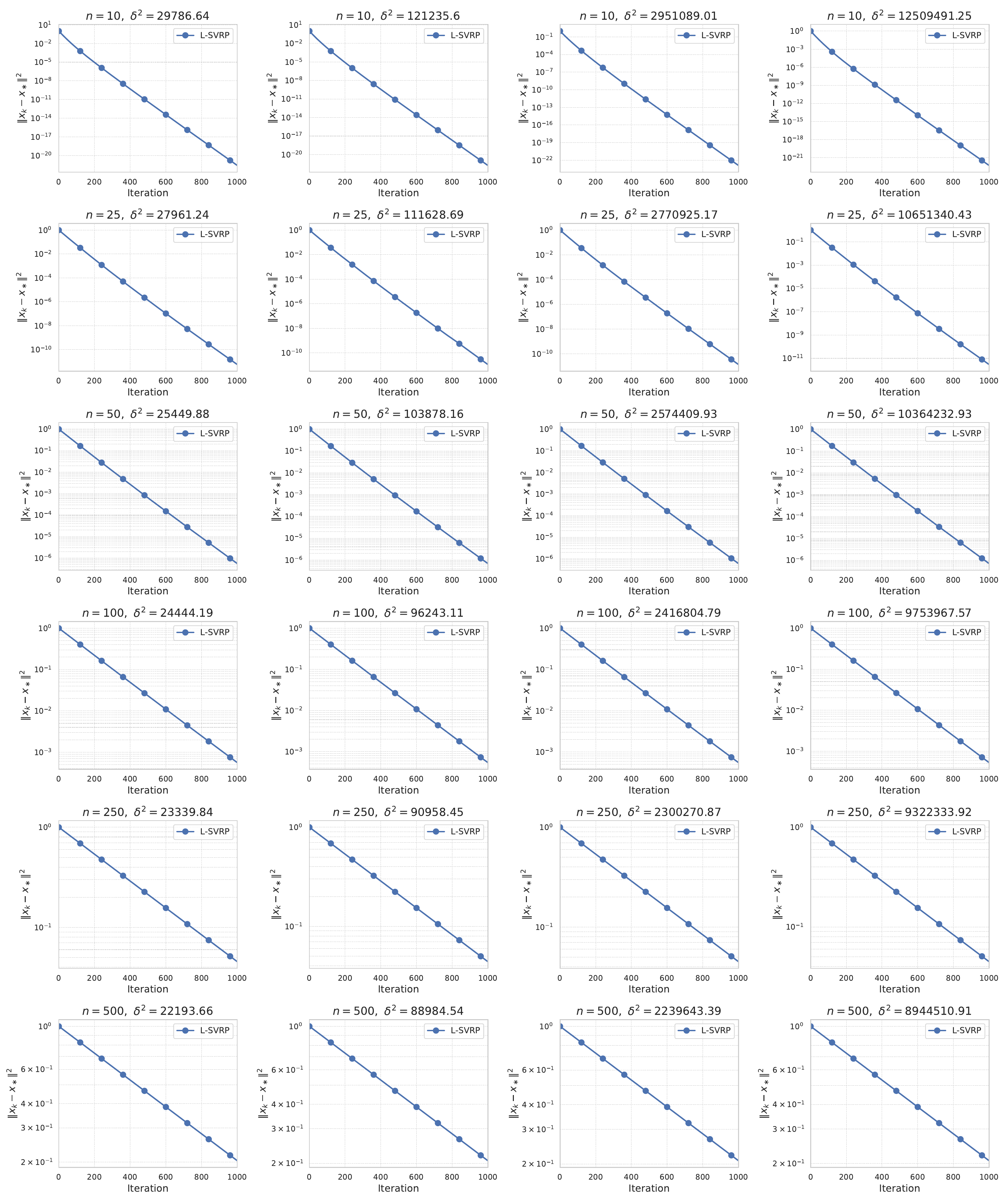}
	\caption{Convergence of~\algname{L-SVRP} for the second half of the 48 configurations described in the experimental section.}
	\label{fig:all_plots_2}	
\end{figure}

\end{document}